\documentclass[11pt]{amsart}
\usepackage{amsmath, cite, tikz}
\usepackage[colorlinks]{hyperref}
\usepackage[capitalize]{cleveref}
\usepackage[a4paper]{geometry}
\usetikzlibrary{calc, decorations.pathmorphing}

\newtheorem{thm}{Theorem}[section]
\newtheorem{conj}[thm]{Conjecture}

\newtheorem{claim}{Claim}
\theoremstyle{definition}
\newtheorem{case}{Case}

\makeatletter \@addtoreset{equation}{section} \makeatother

\def\={\;=\;}
\def\<{\;<\;}
\def\>{\;>\;}

\crefname{def}{Definition}{Definitions}
\Crefname{def}{Definition}{Definitions}
\creflabelformat{def}{~\upshape(#2#1#3)}
\crefname{clm}{Claim}{Claims}
\Crefname{clm}{Claim}{Claims}
\creflabelformat{clm}{~\upshape#2#1#3}
\crefname{ineq}{Ineq.}{Ineqs.}
\Crefname{ineq}{Inequality}{Inequalities}
\creflabelformat{ineq}{~\upshape(#2#1#3)}
\crefname{rl}{Relation}{Relations}
\crefname{rl}{Relation}{Relations}
\creflabelformat{rl}{~\upshape(#2#1#3)}
\crefname{case}{Case}{Cases}
\crefname{case}{Case}{Cases}
\creflabelformat{case}{~\upshape#2#1#3}

\allowdisplaybreaks

\begin{document}
\begin{sloppypar}
\title[An Ore-type Condition for Large $k$-factor and Disjoint Perfect Matchings]
{An Ore-type Condition for Large $k$-factor and Disjoint Perfect Matchings}

\author[H. Lu]{Hongliang Lu}
\address{(Hongliang Lu) School of Mathematics and Statistics, Xi'an Jiaotong University, 710049, Xi'an, P.\ R.\ China}
\email{luhongliang@mail.xjtu.edu.cn}
\author[B. Ning]{Bo Ning}
\address{(Bo Ning\footnote{Corresponding author.}) Center for Applied Mathematics, Tianjin University,
300072, Tianjin, P.\ R.\ China}
\email{bo.ning@tju.edu.cn}
\date{}


\maketitle

\begin{abstract}
Win [\emph{J. Graph Theory} {\bf 6}(1982), 489--492] conjectured
that a graph $G$ on $n$ vertices contains $k$ disjoint perfect
matchings, if the degree sum of any two nonadjacent vertices
is at least $n+k-2$, where $n$ is even and $n\geq k+2$. In
this paper, we prove that Win's conjecture is true for
$k\geq n/2$, where $n$ is sufficiently large. To show this result,
we prove a theorem on $k$-factor in a graph under some Ore-type condition.
Our main tools include Tutte's $k$-factor theorem, the
Karush-Kuhn-Tucker theorem on convex optimization, and the
solution to the longstanding 1-factor decomposition conjecture.
\end{abstract}

\section{Introduction}
To study the existence of a certain type of subgraphs in
a graph is a common topic in graph theory. Maybe the most
well-known theorem is the one proved by Dirac \cite{D52} in 1952,
which is stated as every graph on $n$ vertices has a Hamilton
cycle if every vertex of the graph has degree at least $n/2$. Ore \cite{O60}
extended Dirac's theorem by considering the degree
sum of every pair of nonadjacent vertices in a graph. A graph $G$
is said to be of \textit{Ore-type-(k)} if for every pair of
nonadjacent vertices $x,y$, the degrees of $x,y$ satisfy the
inequality $d(x)+d(y)\geq |G|+k$. Ore \cite{O63} proved that
a graph is \emph{Hamiltonian-connected} if it is of Ore-type-1.
Graphs of Ore-type-$k$ were studied by Roberts \cite{R77}.
Since then, plenty of research was conducted on different graph
properties under Ore-type conditions and the variants,
such as $k$-linkedness \cite{KY08,FGJPPW12}, an equitable coloring
of a graph \cite{KK08}, $k$-ordered Hamiltonicity \cite{FGKLSS03},
and etc. Our note mainly concerns on the existence of disjoint
perfect matchings in a graph under the Ore-type degree condition.

In 1982, Win \cite{W82} posed the following conjecture on disjoint
perfect matchings in a graph of Ore-type-$(k-2)$.

\begin{conj}[Win \cite{W82}]
Let $n,k$ be two integers such that $1\leq k\leq n-2$ and $n$
be even. Let $G$ be a simple graph on $n$ vertices. If $G$ is of Ore-type-$(k-2)$,
then $G$ contains $k$ disjoint perfect matchings.
\end{conj}

For $k=1$, Win's conjecture is true by Ore's theorem \cite{O60}.
Win \cite{W82} further confirmed the conjecture for $k=2,3$.

On the other hand, the existence of perfect matchings in a graph
is closely related to the existence of Hamilton cycles in the same
graph. It is an easy observation that every Hamilton cycle in a graph
corresponds to a pair of disjoint perfect matchings in the graph, if
the order of the graph is even. Egawa  \cite{Egawa93} proved that: Let $k\geq 2$ be
an integer and $G$ be a graph. If $d_G(x)+d_G(y)\geq |G|$ for all non-adjacent
vertices $x,y$, $\delta(G)\geq 2k+1$, and $|G|\geq 8(2k-2)^2$, then $G$
has $k$-edge-disjoint Hamilton cycles. We say that a graph $G$ is a \emph{Fan $2k$-type graph},
if $d(u,v)=2$ implies that $\max\{d(u),d(v)\}\geq n/2+2k$.
Zhou \cite{Z93} conjectured that every $2k$-connected
Fan $2(k-1)$-type graph has $k$ pairwise disjoint Hamilton cycles, and also
confirmed this conjecture for $k=1,2$. Later, the general case $k\geq 3$
was finally finished by Li \cite{L00}. One can easily obtain partial results
on Win's conjecture from the results mentioned above.

However, to the best of our knowledge, Win's conjecture is still wide open now.
One of our results concerns Win's conjecture when $k$ is large in compare with $n$.

\begin{thm}\label{ThWinConj}
Win's conjecture is true for sufficiently large even $n$,
if $k\geq n/2$.
\end{thm}

In this paper, instead of proving Theorem \ref{ThWinConj} directly, we firstly
prove our main result which focuses on the existence of large $k$-factors.

\begin{thm}\label{Thkfactor}
Let $n$ and $k$ be two integers such that $n\geq k+1\geq n/2+1$ and $kn$ be even.
Let $G$ be a graph on $n$ vertices. If $G$ is of Ore-type-$(k-2)$,
then $G$ contains a $k$-factor.
\end{thm}

With the help of Theorem \ref{Thkfactor}, we will use the solution
to 1-factor decomposition conjecture to prove
Theorem \ref{ThWinConj}. Recall that the long-standing 1-factorization conjecture states
that every regular graph of sufficiently large degree has a 1-factorization.
It was first stated explicitly by Chetwynd and Hilton \cite{CH85,CH89}, and they
also stated by Dirac, who discussed it in the 1950s.
Partial results were obtained by Chetwynd and Hilton \cite{CH85,CH89}, and
Zhang and Zhu \cite{ZZ92}. Recently, Csaba et al. \cite{CKLOT} confirmed
this conjecture for large graphs. One of their main results in \cite{CKLOT}
is used for our proof of Theorem \ref{ThWinConj}.

\begin{thm}[Csaba et al. \cite{CKLOT}]\label{Th:Csabeetal}
Suppose that $n$ is sufficiently large and even, and $D\geq 2\lceil n/4\rceil-1$.
Then every $D$-regular graph $G$ on n vertices has a decomposition into
perfect matchings.
\end{thm}

The proof of our main theorem also uses a theorem of Katerinis and Woodall on
$k$-factor, and the Karush-Kuhn-Tucker theorem on convex optimization. We will
introduce all necessary terminology and additional
results in the next section.

Now we give some necessary notation and terminology. Let $G$ be a graph. We use $V(G)$ and
$E(G)$ to denote the vertex set and edge set of $G$, respectively,
and denote by $|G|=|V(G)|$. Let $S,T$ be two disjoint subsets of
$V(G)$, $E_G(S,T)$ be the set of edges between $S$ and $T$
in $G$, and $e_G(S,T)=|E_G(S,T)|$. When $S$ consists of a single
element, say $S=\{v\}$, we use $E_G(v,T)$ and $e_G(v,T)$
instead of $E_G(\{v\},T)$ and $e_G(\{v\},T)$, respectively.
Let $v\in V(G)$ and $H$ be a subgraph of $G$. $N_G(v)$ is the set
of neighbors of $v$ in $G$ and $d_G(v)=|N_G(v)|$. Set $N_H(v)=N_G(v)\cap V(H)$
and $d_H(v)=|N_H(v)|$. When there is no danger of ambiguity, we use
$d(v)$ instead of $d_G(v)$ for short. Let $S\subset V(G)$ and let $G-S$ denote the
subgraph of $G$ induced by $V(G)\backslash V(S)$. If $S$ consists of
only one vertex, say $S=\{v\}$, we use $G-v$ instead of $G-\{v\}$.
For notation and terminology not defined here, we refer the reader to
Bondy and Murty \cite{BM08}.

The organization of our paper is as follows. In Section 2, we
introduce necessary preliminaries.
In Section 3, we prove Theorems \ref{ThWinConj} and \ref{Thkfactor}.

\section{Some preliminaries}
In this section, we first introduce some notation and
terminology related to Tutte's $k$-factor
theorem. For any pair of disjoint subsets $S,T\subset V(G)$, a component
$C$ of $G-S-T$ is called a \textit{$k$-odd-component}
if
$$
e_G(V(C),T)+k|V(C)|\equiv 1~~~\pmod 2.
$$
We usually use $q(S,T)$ to denote the number of components
of $G-S-T$ which are $k$-odd components.

Tutte's $k$-factor theorem is well known.
\begin{thm}[Tutte \cite{T52}]\label{Tutte52}
Let $k$ be a positive integer. A graph $G$ contains no $k$-factor
if and only if there exist disjoints subsets $S,T\subset V(G)$,
such that
\begin{align}
\eta(S,T):=k|S|-k|T|+\sum_{x\in T}d_{G-S}(x)-q(S,T)\leq -2. \label{con:1}
\end{align}
\end{thm}

From Tutte's theorem, Katerinis and Woodall \cite{KW87} deduced the following.
It shall play an important role in our proof.

\begin{thm}[Katerinis and Woodall \cite{KW87}]\label{factor-prop}
Let $k\geq 1$ be an integer. If a graph $G$ contains no $k$-factor, then
there exist two disjoint subsets $S,T\subset V(G)$ such that there holds (2.1),
and
\begin{align}
& e_G(v,T)\leq k-1,~and \label{con:2}\\
& d_{G-S}(v)\geq k+1~for~all~v\in U,\label{con:3}
\end{align}
where $U$ denotes the union of all $k$-odd components of $G-S-T$.

\end{thm}

Our proof also uses tools from optimization.
An optimization problem of the form
\begin{equation}\label{con-op-pr}
 \left
\{\begin{array}{ll}
              \min \quad f(x),  & \hbox{}\\
             s.t.\quad\ g_i(x)\leq 0, & for~i=1,\cdots,m\hbox{}
             \end{array}
           \right.
\end{equation}
is called a \emph{convex optimization problem} if the functions
$f, g_1,\ldots, g_m: R_n\rightarrow R$ are all convex functions.
We need the Karush-Kuhn-Tucker theorem on convex optimization.
The following one is a direct corollory of Theorem 4.3.8 in \cite[pp.207]{BSS06}.

\begin{thm}[Karush-Kuhn-Tucker sufficient condition \cite{BSS06}]\label{Th-KKT}
Let $X$ be a nonempty open set in $R^n$, and let $f: R^n\longrightarrow R$,
$g_i:R^n\longrightarrow R$ for $i=1,\cdots,m$. Consider Problem P:
\begin{equation}\label{NLP}
 \left
\{\begin{array}{ll}
              \min \quad f(x),  & \hbox{}\\
             s.t.\quad\ g_i(x)\leq 0, & for~i=1,\cdots,m\hbox{} \\
              \quad\quad\ x\in X, & \hbox{}
             \end{array}
           \right.
\end{equation}
Let $\overline{x}$ be a local optimal solution. There exist scalars $u_i\geq 0$ for $1\le i\le m$
such that
\begin{align}\label{KKT-point}
\nabla f(\overline{x})+\sum_{i\in I}u_i\nabla g_i(\overline{x})=0.
\end{align}%

\end{thm}
The point satisfying (\ref{KKT-point}) is called a \emph{KKT point}.
For convex optimal problems, the KKT conditions are also sufficient
for optimality (see \cite[pp.773]{BSS06}.).
\begin{thm}\label{con-op}
For the convex optimal problem (\ref{con-op-pr}), every KKT point
is a global optimal solution.
\end{thm}

The next result is a well-known result on convex function.
\begin{thm}\label{Th-Convex}
Let $f(x)$ be a function on $R$, where $R$ is a convex set. Suppose that $f$
is twice differentiable and $f''$ is continuous. Then $f(x)$ is a convex function
if and only if its Hessian matrix is positive semi-definite on $R$.
\end{thm}

For more information and details, we refer the reader to Boyd and Vandenberghe \cite{BV13}.

\section{Proofs of Theorems \ref{ThWinConj} and \ref{Thkfactor}.}
In this section, we will present the proofs of Theorems \ref{ThWinConj} and \ref{Thkfactor}.

\noindent
{\bf Proof of Theorem \ref{Thkfactor}.} We prove Theorem \ref{Thkfactor} by contradiction.
Suppose that $G$ contains no $k$-factors. By Theorem \ref{factor-prop}, we can choose
disjoint $S,T\subset V(G)$ satisfying (\ref{con:1}), (\ref{con:2}), and  (\ref{con:3}).
Define $s:=|S|$ and $t:=|T|$. Let $C_1,\dots, C_{q}$ be all $k$-odd components of $G-S-T$.
So, for every vertex $v\in V(C_i)$, $d_{G-S}(v)\geq k+1$
and $e_{G}(v,T)\leq k-1$, and this implies $d_{C_i}(v)\geq 2$. Thus, $|C_i|\geq 3$.

\begin{claim}\label{Claim-kcon}
$G$ is $k$-connected, and hence, the minimum degree $\delta(G)\geq k$.
\end{claim}
\begin{proof}
Let $W$ be a cut-set of $G$ and let $C_1',C_2'$ be two components of $G-W$.
For $x\in V(C_1')$ and $y\in V(C_2')$, one can see that $xy\notin E(G)$, and thus
\[
n+k-2\leq d(x)+d(y)\leq |C_1'|+|C_2'|-2+2|W|.
\]
Notice that $n\geq |C_1'|+|C_2'|+|W|$. Hence, $|W|\geq k$,
and moreover, $\delta(G)\geq k$.
\end{proof}
Now we show that $T\neq \emptyset$. Otherwise, by
(\ref{con:1}) and Claim \ref{Claim-kcon}, we have
$q(S,\emptyset)\geq ks+2\geq k^2+2$. Thus,
$n\geq |U|+s+t\geq 3(k^2+2)+k\geq \frac{3}{4}n^2+\frac{1}{2}n+6$, which is impossible.

Set $h_1:=\min\{d_{G-S}(x):x\in T\}$. Let $u_1\in T$ such that
$d_{G-S}(u_1)=h_1$. Set $N_T[u_1]:=(N(u_1)\cap T)\cup\{u_1\}$.
For any vertex $x\in V(G)$, let $d_T(x)=|N_G(x)\cap T|$. If
$T-N_T[u_1]\neq \emptyset$, let $h_2:=\min\{d_{G-S}(x):x\in T-N_{T}[u_1]\}$
and choose $u_2\in T-N_T[u_1]$ such that $d_{G-S}(u_2)=h_2$.
As in the proof of Lemma \ref{factor-prop},
we still denote $U:=C_1\cup C_2\ldots \cup C_q$.

\begin{claim}
\begin{align}\label{cl:sh1k}
s+h_1\geq k.
\end{align}
\end{claim}
\begin{proof}
Since $\delta(G)\ge k$, $s+h_1\ge d_{G}(u_1)\ge k$.
\end{proof}

In the following, we divide the proof into four
cases.

\case $h_1\geq k$.

By (\ref{con:1}), we have
\begin{align*}
q:=q(S,T)\geq k|S|-k|T|+2+\sum_{x\in T}d_{G-S}(x)\geq k|S|-k|T|+2+h_1|T|\geq ks+2\geq 2.
\end{align*}
This means that $G-S-T$ is disconnected. By Claim \ref{Claim-kcon}, $s+t\geq k$.
Notice that $k\geq n/2$. Since $|C_i|\geq 3$ for each $i=1,\ldots,q$, we infer that
$|U|\geq 3q\geq3( ks+2)$. If $|S|\geq 1$, then $n=|G|\geq|U|+s+t\geq 3(k+2)+s+t\geq 4k+6>n$,
a contradiction. Thus, $S=\emptyset$ and $t\geq k$. Since $q\geq 2$,
choose $x\in V(C_1)$ and $y\in V(C_2)$, and we have
\begin{align}
&n+k-2\nonumber\\
&\leq d(x)+d(y)\nonumber\\
&\leq |C_1|-1+|N_G(x)\cap T|+|C_2|-1+|N_G(y)\cap T|\nonumber\\
&\leq |C_1|+|C_2|+2k-4~~~\quad \mbox{(by~Theorem~\ref{factor-prop})}\nonumber\\
&\leq n-t+2k-4\nonumber\\
&\leq n+k-4,\nonumber
\end{align}
a contradiction.

Thus, in the following, assume that
\begin{align}\label{Case-1}
h_1\leq k-1.
\end{align}

\case $T=N_T[u_1]$.

\begin{claim}\label{non-adja}
For any $i\in \{1,\ldots,q\}$, there
exists $w_i\in V(C_i)$ such that $w_iu_1\notin E(G)$
\end{claim}
\begin{proof}
Suppose that there exists $j\in \{1,\ldots,q\}$,
such that $V(C_j)\subset N_{G-S}(u_1)$. Notice that
for $x\in V(C_j)$, $d_{G-S}(x)\geq k+1$, and
$N_{G-S}(x)\subset V(C_j)\cup T$. Then
by (\ref{Case-1}), $k-1\geq h_1=d_{G-S}(u_1)\geq |C_j|+|T|-1\geq d_{G-S}(x)\geq k+1$,
a contradiction.
\end{proof}

\begin{claim}
\begin{align}\label{bound-c_i}
|C_i|\geq k-h_1+2.
\end{align}
\begin{align}\label{bound-n}
n\geq s+t+q(k-h_1+2).
\end{align}
\end{claim}

\begin{proof}
For each $i\in \{1,2,\ldots,q\}$, by Claim \ref{non-adja}, there exists a vertex
$x_i\in V(C_i)$ such that $x_iu_1\notin E(G)$. Since
$d_T(x_i)\leq |T|-1= d_T(u_1)\leq d_{G-S}(u_1)=h_1$,
we have
\begin{align*}
|C_i|&\geq |N_G(x_i)\cap V(C_i)|+1\\
&=d_{G-S}(x_i)-d_{T}(x_i)+1\\
&\geq (k+1)-h_1+1\\
&=k-h_1+2.
\end{align*}
Moreover, by (\ref{bound-c_i}), we can get
\begin{align*}
n=|G|\geq |S|+|T|+\sum_{i=1}^q|C_i|\geq s+t+q(k-h_1+2).
\end{align*}
\end{proof}

\begin{claim}\label{claim-q}
$q=q(S,T)\geq 2$.
\end{claim}
\begin{proof}
By (\ref{Case-1}), the fact $T=N_T[u_1]$, and the
definition of $u_1$, we infer
\begin{align*}
d_T(u_1)=t-1\leq d_{G-S}(u_1)=h_1\leq k-1,
\end{align*}
Thus, $k\geq h_1+1\geq t$. So,
\begin{align}
q(S,T)&\geq ks-kt+\sum_{x\in T}d_{G-S}(x)+2\nonumber\\
&\geq ks-kt+h_1t+2\label{kskt}\\
&\geq k(k-h_1)-kt+h_1t+2\nonumber\\
&=(k-h_1)(k-t)+2\nonumber\\
&\geq 2.\label{2}
\end{align}
\end{proof}

\begin{claim}
\begin{align}\label{1s-bound2}
s\geq k+(q-1)(k+2-h_1)+t- 2h_1-1.
\end{align}
\end{claim}

\begin{proof}
For any $i\in \{1,\ldots,q\}$,
since $w_iu_1\notin E(G)$, we have
\begin{align}\label{wu11}
d(w_i)+d(u_1)\geq n+k-2.
\end{align}
On the other hand, we obtain
$$
d_T(w_i)\leq t-1=d_T(u_1)\leq d_{G-S}(u_1)=h_1.
$$
One can see that
\begin{align}\label{wu12}
d(w_i)+d(u_1)\leq |C_i|-1+2h_1+2s.
\end{align}

Combining (\ref{wu11}) and (\ref{wu12}), we can infer
\begin{align*}
&n+q(2h_1+2s-1)\\
&\geq \sum_{i=1}^q|C_i|+q(2h_1+2s-1)+s+t\\
&\geq q(n+k-2)+s+t,
\end{align*}
that is,
\begin{align*}
(q-1)n\leq q(2h_1+2s-1-k+2)-s-t.
\end{align*}
By Claim \ref{claim-q}, $q\geq 2$.
By $(\ref{bound-n})$, we have
\begin{align*}
(q-1)(s+t+q(k-h_1+2))\leq q(2h_1+2s-k+1)-s-t.
\end{align*}
This implies
\begin{align*}
(q-1)(k+2-h_1)+s+t\leq 2h_1+2s-k+1,
\end{align*}
and this proves the claim.
\end{proof}

By computation, we have
\begin{align}
0&\geq 2+ks-kt+\sum_{x\in T}d_{G-S}(x)-q\nonumber\\
&\geq 2+k(k+(q-1)(k+2-h_1)- 2h_1-1)+h_1t-q\quad\nonumber \mbox{(by~(\ref{1s-bound2}))}\\
&=2+q(k(k+2-h_1)-1)-k(3+h_1)+h_1t\nonumber\\
&\geq 2+(ks-kt+h_1t+2)(k(k+2-h_1)-1)-k(3+h_1)+h_1t\quad\nonumber \mbox{(by~(\ref{Case-1}) and (\ref{kskt}))}\\
&\geq 2+(ks-(k-h_1)(h_1+1)+2)(k(k+2-h_1)-1)\break\nonumber \\
&-k(3+h_1)+h_1(h_1+1)\nonumber\\
&\geq 2+2(k(k+2-h_1)-1)-k(3+h_1)+h_1(h_1+1)\nonumber \\
&=2k^2+k-3kh_1+h_1^2+h_1\nonumber\\
&\geq 3k\nonumber,
\end{align}
where we have used the fact $k(k+2-h_1)-1\geq 3k-1\geq 0$ in the third inequality above;
and (\ref{cl:sh1k}), (\ref{Case-1}) and (\ref{kskt}) in the fifth inequality
above; and the fact $f(h_1)\geq f(k-1)$ in the last step,
where the function $f(x)=-3kx+x^2+x$, $x\leq k-1$.

This contradiction completes the proof of the case.

\case $T\neq N_T[u_1]$ and $h_2\geq k$.

Set $p:=|N_T[u_1]|$. Recall that $V(U)=V(C_1\cup \ldots \cup C_q)$.
We have

\begin{claim}\label{claim-q2}
$q(S,T)\geq 2$, where the
equality holds when $h_1=k-1$, $p=k$
and $h_2=k$.
\end{claim}
\begin{proof}
By (\ref{con:1}), we have
\begin{align*}
q(S,T)&\geq k|S|-k|T|+\sum_{x\in T}d_{G-S}(x)+2\\
&\geq ks-kt+h_1p+h_2(t-p)+2.
\end{align*}
By the hypothesis $h_2\geq k$ and $t\geq p$, we obtain
\begin{align}\label{kst}
ks-kt+h_1p+h_2(t-p)+2
&\geq ks-(k-h_1)p+2.
\end{align}
By (\ref{cl:sh1k}), $s\geq k-h_1$. Since $s\geq k-h_1$
and $p\leq h_1+1$, we obtain $ks-(k-h_1)p+2\geq (k-h_1)(k-h_1-1)+2$.
By (\ref{Case-1}), $k\geq h_1+1$. So, $(k-h_1)(k-h_1-1)+2\geq 2$,
and hence $q(S,T)\geq 2$. The condition when the equality holds
can be deduced easily. This proves the claim.
\end{proof}

\begin{claim}\label{cl-s}
Suppose there exists a vertex $x\in V(U)$
such that $xu_1\notin E(G)$. Then
\begin{align}\label{3s-bound1}
s\geq k+3(q-1)-h_1.
\end{align}
\end{claim}

\begin{proof}
Without loss of generality, assume that $x\in V(C_i)$ for some
$i\in \{1,\ldots,q\}$. By Lemma \ref{factor-prop}, one may see
that $|C_i|\geq 3$ for $i=1,\ldots,q$. So we obtain
\begin{align*}
n\geq s+t+3q.
\end{align*}
We also have
\begin{align*}
n+k-2\leq d(x)+d(u_1)\leq (|C_i|-1)+(t-1)+s+h_1+s=2s+t+|C_i|+h_1-2.
\end{align*}
One can see that
\begin{align*}
|C_i|\leq n-s-t-3(q-1).
\end{align*}
Thus, we have
\begin{align*}
n+k-2\leq s+n+h_1-3(q-1)-2.
\end{align*}
This proves the claim.
\end{proof}

\begin{claim}\label{claim-Uu1}
$V(U)\subset N_G(u_1)$.
\end{claim}
\begin{proof}
Suppose not. By Claim \ref{cl-s}, (\ref{3s-bound1}) holds.
Thus,
\begin{align*}
0&\geq 2+ks-kt+\sum_{x\in T}d_{G-S}(x)-q\\
&\geq 2+k(3(q-1)+k-h_1)-kt+h_1p+h_2(t-p)-q  \quad \mbox{(by~(\ref{3s-bound1}))}\\
&\geq 2+k(3(q-1)+k-h_1)+(h_1-k)(h_1+1)-q\\
&\geq k(3+k-h_1)+(h_1-k)(h_1+1)\\
&=h_1^2-(2k-1)h_1+k^2+2k\\
&\geq 3k,
\end{align*}
a contradiction. Notice that in the above, we have used the facts
$h_2\geq k$, $t\geq p$, $h_1\leq k-1$ and $p\leq h_1+1$
in the third step; and the facts that the function $f(q)=3k(q-1)-q$
is increasing and $q\geq 2$ (by Claim \ref{claim-q2})
in the fourth step; and the fact that the function
$f(h_1)=h_1^2-(2k-1)h_1+k^2+2k$ is decreasing when $h_1\leq k-1$ in
the last step.

The proof of this claim is complete.
\end{proof}
By Claim \ref{claim-Uu1}, $V(U)\subset N_G(u_1)$.
So, $h_1\geq 3q+p-1$. We have
\begin{align*}
0&\geq 2+ks-kt+\sum_{x\in T}d_{G-S}(x)-q\\
&\geq 2+k(k-h_1)-kt+h_1p+h_2(t-p)-q\quad \mbox{(by~the~fact~$s+h_1\geq k$)}\\
&\geq 2+k(k-h_1)+(h_1-k)p-q\quad \mbox{(by~(\ref{kst}))}\\
&\geq 2+(k-h_1)(k-p)-q\\
&\geq 2+(k-h_1)(k-h_1+3q-1)-q\\
&\geq (k-h_1)(k-h_1+5)\quad \mbox{(since $q\ge 2$)}\\
&>0,
\end{align*}
a contradiction. This proves the case.

\case  $0\leq h_1\leq h_2\leq k-1$.

Since $u_1u_2\notin E(G)$, it follows that
\begin{align*}
n+k-2\leq d(u_1)+d(u_2)\leq h_1+h_2+2s,
\end{align*}
i.e.,
\begin{align}\label{4s-bound1}
s\geq \frac{1}{2}(n+k-2-h_1-h_2).
\end{align}
Since $|C_i|\geq 3$, one may see that
\begin{align}\label{4n-bound}
n\geq s+t+3q.
\end{align}
We can get
\begin{align*}
0&\geq ks-kt+h_1p+h_2(t-p)+2-q\\
&=ks-(k-h_2)t+(h_1-h_2)p+2-q\\
&\geq ks-(k-h_2)(n-s-3q)+(h_1-h_2)p+2-q\quad \mbox{(by (\ref{4n-bound}))}\\
&\geq (2k-h_2)s-(k-h_2)n+q(3(k-h_2)-1)+(h_1-h_2)(h_1+1)+2\mbox{~(since~$p\leq h_1+1,h_1\leq h_2$)}\\
&\geq (2k-h_2)s-(k-h_2)n+(h_1-h_2)(h_1+1)+2,
\end{align*}
i.e.,
\begin{align}\label{4tutte}
0\geq (2k-h_2)s-(k-h_2)n+(h_1-h_2)(h_1+1)+2.
\end{align}

First suppose that
\begin{align}\label{h1-h2}
h_1-h_2\geq k+2-n.
\end{align}
One can see that
\begin{align*}
0&\geq \frac{1}{2}(2k-h_2)(n+k-2-h_1-h_2)-(k-h_2)n+(h_1-h_2)(h_1+1)+2\quad \mbox{(by (\ref{4s-bound1}))}\\
&=h_1^2-h_1(k-1+\frac{h_2}{2})+\frac{h_2^2}{2}+\frac{1}{2}(n-3k)h_2+k^2-2k+2\\
&\geq h_1^2-h_1(k-1+\frac{h_2}{2})+\frac{h_2^2}{2}+\frac{1}{2}\left((k+2+h_2-h_1)-3k\right)h_2+k^2-2k+2\ \mbox{~~~(by (\ref{h1-h2}))}\\
&=h_1^2-h_1(k-1+h_2)+h_2^2+(-k+1)h_2+k^2-2k+2,
\end{align*}
i.e.,
\begin{align}\label{0>=f}
0\geq h_1^2-h_1(k-1+h_2)+h_2^2+(-k+1)h_2+k^2-2k+2.
\end{align}
Let $f(h_1,h_2,k)=h_1^2-h_1(k-1+h_2)+h_2^2+(-k+1)h_2+k^2-2k+2$.
Consider the following non-linear programming problem:
\begin{equation}\label{NLP}
 \left
\{\begin{array}{ll}
              \min \quad f(h_1,h_2,k),  & \hbox{}\\
             s.t.\quad  h_1-h_2\leq 0, & \hbox{} \\
               \quad\quad\ h_2\leq k-1, & \hbox{}\\
              \quad\quad\ -h_1\leq 0, & \hbox{}
             \end{array}
           \right.
\end{equation}
The Hessian matrix of the function $f(h_1,h_2,k)$ is
\begin{equation*}
M=\left(
    \begin{array}{ccc}
      2 & -1 & -1 \\
      -1 & 2 & -1 \\
      -1 & -1 & 2 \\
    \end{array}
  \right).
\end{equation*}
Note that $M$ is a positive semi-definite matrix. By
Theoerem \ref{Th-Convex}, $f(h_1,h_2,k)$ is a convex function. Thus
(\ref{NLP}) is a convex optimization problem. Its Lagrangian function is
\begin{align*}
L(h,\lambda)=&h_1^2-h_1(k-1+h_2)+h_2^2+(-k+1)h_2+k^2-2k+2+\lambda_1(h_1-h_2)\\
&+\lambda_2(h_2-k+1)+\lambda_3(-h_1).
\end{align*}
Hence the Karush-Kuhn-Tucker condition of (\ref{NLP}) is
\begin{equation}\label{KT-Eq}
  \left\{
     \begin{array}{ll}
        2h_1- (k-1+h_2)+\lambda_1-\lambda_3=0, & \hbox{} \\
       -h_1+2h_2+(-k+1)-\lambda_1+\lambda_2=0, & \hbox{} \\
       -h_1-h_2+2k-2-\lambda_2=0, & \hbox{} \\
       \lambda_1(h_1-h_2)=0, & \hbox{} \\
       \lambda_2(h_2-k+1)=0, & \hbox{} \\
       \lambda_3h_1=0. &
     \end{array}
   \right.
\end{equation}

It is easy to see that $h_1=h_2=k-1$ and $\lambda_1=\lambda_2=\lambda_3=0$
is a solution of the equation (\ref{KT-Eq}).
For a convex optimization problem, by Theorem \ref{Th-KKT},
every solution satisfying its KKT condition is also its
optimum solution. Thus, we have
\begin{align*}
  f(h_1,h_2,k)\ge f(k-1,k-1,k)=1,
\end{align*}
contradicting $(\ref{0>=f})$.

Finally, suppose that
\begin{align*}
h_2-h_1>n-k-2.
\end{align*}

By (\ref{4tutte}) and (\ref{cl:sh1k}), one can see that
\begin{align*}
0&\geq (2k-h_2)s-(k-h_2)n+(h_1-h_2)(h_1+1)+2\\
&\geq (2k-h_2)(k-h_1)-(k-h_2)n+(h_1-h_2)(h_1+1)+2\\
&=h_1^2-(2k-1)h_1+2k^2-kn+2+h_2(n-k-1)\\
&\geq h_1^2-(2k-1)h_1+2k^2-kn+2+(n-k-1)^2+(n-k-1)h_1\\
&=h_1^2-(3k-n)h_1+2k^2-kn+2+(n-k-1)^2\\
&\geq -\frac{1}{4}(3k-n)^2+2k^2-kn+2+(n-k-1)^2\\
&=-\frac{1}{4}(n-k)^2+(n-k-1)^2+2\\
&=\frac{3}{4}(n-k)^2-2(n-k)+3\\
&>0,
\end{align*}
a contradiction.
This completes the proof of Theorem
\ref{Thkfactor}.  {\hfill$\Box$}

\noindent
{\bf Proof of Theorem \ref{ThWinConj}.} By Theorem \ref{Thkfactor},
$G$ contains a $k$-factor, denoted by $H$, where
$k\geq n/2\geq 2\lceil n/4\rceil-1$. Obviously, $H$ is
$k$-regular. Since the order of $G$ is sufficiently large,
the order of $H$ is also sufficiently large. By Theorem \ref{Th:Csabeetal},
$H$ can be decomposed into $k$ disjoint perfect matchings.
The proof of Theorem \ref{ThWinConj} is completed.  {\hfill$\Box$}

\section*{Acknowledgements}
The first author is supported by NSFC (No.\ 11471257).
The second author is supported by NSFC (Nos.\ 11601379 and \ 11971346).

\end{sloppypar}

\begin{thebibliography}{99}
\bibitem{BSS06}
M.S. Bazaraa, H.D. Shrali, C.M. Shetty, Nonlinear Programming Theory and Algorithms,
Wiley, 2006.

\bibitem{BM08}
J.A. Bondy, U.S.R. Murty, Graph Theory, GTM 244, Springer, 2008.

\bibitem{BV13}
S. Boyd, L. Vandenberghe, Convex Optimazation, Cambridge
Unversity Press \& Beijing World Publishing Corporation, 2013.

\bibitem{CH85}
A.G. Chetwynd, A.J.W. Hilton,
Regular graphs of high degree are $1$-factorizable,
\emph{Proc. London Math. Soc.} {\bf 50}(1985), 193--206.

\bibitem{CH89}
A.G. Chetwynd, A.J.W. Hilton, 1-factorizing regular graphs of high degree¡ªan improved
bound, \emph{Discrete Math.} {\bf 75}(1989), 103--112.

\bibitem{CKLOT}
B. Csaba, D. K\"{u}hn, A. Lo, D. Osthus, A. Treglown,
Proof of the 1-factorization and Hamilton
decomposition conjectures, \emph{Mem. Amer. Math. Soc.} {\bf 244}(2016),
no. 1154, v+164 pp. ISBN: 978-1-4704-2025-3; 978-1-4704-3508-0; J12.

\bibitem{D52}
G.A. Dirac, Some theorems on abstract graphs, \emph{Proc. London.
Math. Soc.} {\bf 2} (1952) 69--81.

\bibitem{Egawa93}
Y. Egawa, Edge-disjoint Hamiltonian cycles in graphs of Ore-type,
\emph{SUT J. Math.} {\bf 29} (1993), 15--50.

\bibitem{FGKLSS03}
R.J. Faudree, R.J. Gould, A.V. Kostochka, L. Lesniak, I. Schiermeyer, A. Saito,
Degree conditions for k-ordered Hamiltonian graphs,
\emph{J. Graph Theory} {\bf 42}(2003), no. 3, 199--210.

\bibitem{FGJPPW12}
M. Ferrara, R. Gould, M. Jacobson, F. Pfender, J. Powell, T. Whalen,
New Ore-type conditions for {$H$}-Linked graphs, \emph{J. Graph Theory} {\bf 71}(2012),
no. 1, 69--77.

\bibitem{KW87}
P. Katerinis, D.R. Woodall, Binding numbers of graphs and the
existence of $k$-factors, \emph{Quart. J. Math.}
{\bf 38}, (1987), 221--228.

\bibitem{KK08}
H.A. Kierstead, A.V. Kostochka,
An Ore-type theorem on equitable coloring,
\emph{J. Combin. Theory Ser. B} {\bf 98}(2008), no. 1, 226--234.

\bibitem{KY08}
A.V. Kostochka, G.X. Yu, Ore-type degree conditions for a graph to be $H$-linked,
\emph{J. Graph Theory} {\bf 58}(2008), no. 1, 14--26.

\bibitem{L00}
G.J. Li, Edge disjoint Hamilton cycles in graphs,
\emph{J. Graph Theory} {\bf 35}(2000), no. 1, 8--20.

\bibitem{O60}
O. Ore, Note on Hamilton circuit, \emph{Am. Math. Mon.} {\bf 67}(1960) 55.

\bibitem{O63}
O. Ore, Hamilton connected graphs, \emph{J. Math. Pures Appl.} {\bf 42} (1963), 21--27.

\bibitem{R77}
J. Roberts, Hamiltonian properties in the graphs of Ore-type-($k$),
\emph{Bull. London Math. Soc.} {\bf 9}(1977), 295--298.

\bibitem{T52}
W.T. Tutte, The factors of graphs, \emph{Canad. J. Math.} {\bf 4}(1952),
314--328.

\bibitem{W82}
S.~Win, A sufficient condition for a graph to contain three disjoint 1-factors,
\emph{J. Graph Theory} {\bf 6}(1982), 489--492.

\bibitem{ZZ92}
C.Q. Zhang, Y.J. Zhu,
Factorizations of regular graphs,
\emph{J. Combin. Theory Ser. B} {\bf 56}(1992), 74--89.

\bibitem{Z93}
S.M. Zhou, Disjoint Hamilton cycles in Fan-$2k$ type graphs,
\emph{J. Graph Theory} {\bf 17}(1993), 673--678.
\end{thebibliography}
\end{document}